\documentclass[11pt]{amsart}
\usepackage{amsfonts}
\usepackage{amsmath}
\usepackage{amssymb}
\usepackage{txfonts}
\usepackage{color}
\usepackage{graphicx}
\usepackage{xspace}
\usepackage{axodraw}
\input xy
\xyoption{all}
 \font \eightrm=cmr8
\def\diagramme #1{\vskip 4mm \centerline {#1} \vskip 4mm}

 \newcommand{\nc}{\newcommand}

 \nc{\butcher}{{\scriptstyle\ \circleright\ }}
 \nc{\dcdot}{{\raise 0.5pt\hbox{\bf .}}}
 
\nc{\surj}{\to\hskip -3mm \to}
\nc{\urhd}{\,\underline\rhd\,}
\nc{\ulhd}{\,\underline\lhd\,}

\newtheorem{thm}{Theorem}

\newtheorem{cor}[thm]{Corollary}

\newtheorem{lem}[thm]{Lemma}
\newtheorem{prop}[thm]{Proposition}

\newtheorem{rmk}[thm]{Remark}

\newcommand{\tree}{\hskip 0.8pc\scalebox{-0.3}{{\parbox{0.5pc}{
  \begin{picture}(30,45) (75,-60)
    \SetWidth{1.5}
    \SetColor{Black}
    \Line(90,-30)(75,-60)
    \Line(90,-30)(105,-60)
    \Line(90,-15)(90,-30)
  \end{picture}}}}}

\newcommand{\treeA}{\hskip 1.5pc\scalebox{-0.3}{{\parbox{0.5pc}{
   \begin{picture}(60,75) (75,-30)
    \SetWidth{1.5}
    \SetColor{Black}
    \Line(90,0)(75,-30)
    \Line(90,0)(105,-30)
    \Line(105,30)(90,0)
    \Line(105,30)(135,-30)
    \Line(105,45)(105,30)
  \end{picture}}}}}

\newcommand{\treeB}{\hskip 1.5pc\scalebox{-0.3}{{\parbox{0.5pc}{
  \begin{picture}(60,75) (75,-30)
    \SetWidth{1.5}
    \SetColor{Black}
    \Line(90,0)(75,-30)
    \Line(105,30)(135,-30)
    \Line(105,45)(105,30)
    \Line(120,0)(105,-30)
    \Line(105,30)(90,0)
  \end{picture}}}}}

\newcommand{\treeC}{\hskip 1.5pc\scalebox{-0.2}{{\parbox{0.5pc}{
 \begin{picture}(90,105) (75,-30)
    \SetWidth{2.9}
    \SetColor{Black}
    \Line(90,0)(75,-30)
    \Line(120,60)(90,0)
    \Line(90,0)(105,-30)
    \Line(105,30)(135,-30)
    \Line(120,60)(165,-30)
    \Line(120,75)(120,60)
  \end{picture}
}}}}

\newcommand{\treeD}{\hskip 1.5pc\scalebox{-0.2}{{\parbox{0.5pc}{
 \begin{picture}(90,105) (75,-30)
    \SetWidth{2.9}
    \SetColor{Black}
    \Line(90,0)(75,-30)
    \Line(120,60)(90,0)
    \Line(90,0)(105,-30)
    \Line(120,60)(165,-30)
    \Line(120,75)(120,60)
    \Line(150,0)(135,-30)
  \end{picture}
}}}}

\newcommand{\treeE}{\hskip 1.5pc\scalebox{-0.2}{{\parbox{0.5pc}{
 \begin{picture}(90,105) (75,-30)
    \SetWidth{2.9}
    \SetColor{Black}
    \Line(90,0)(75,-30)
    \Line(120,60)(90,0)
    \Line(120,60)(165,-30)
    \Line(120,75)(120,60)
    \Line(150,0)(135,-30)
    \Line(135,30)(105,-30)
  \end{picture}
}}}}

\newcommand{\treeF}{\hskip 1.5pc\scalebox{-0.2}{{\parbox{0.5pc}{
 \begin{picture}(90,105) (75,-30)
    \SetWidth{2.9}
    \SetColor{Black}
    \Line(90,0)(75,-30)
    \Line(120,60)(90,0)
    \Line(120,60)(165,-30)
    \Line(120,75)(120,60)
    \Line(135,30)(105,-30)
    \Line(120,0)(135,-30)
  \end{picture}
}}}}

\newcommand{\treeG}{\hskip 1.5pc\scalebox{-0.2}{{\parbox{0.5pc}{
 \begin{picture}(90,105) (75,-30)
    \SetWidth{2.9}
    \SetColor{Black}
    \Line(90,0)(75,-30)
    \Line(120,60)(90,0)
    \Line(120,60)(165,-30)
    \Line(120,75)(120,60)
    \Line(105,30)(135,-30)
    \Line(120,0)(105,-30)
  \end{picture}
}}}}

\newcommand{\treebc}{\scalebox{0.21}{{
    \begin{picture}(100,105) (100,-90)
    \SetWidth{2.0}
    \SetColor{Black}
    \Line(150,-135)(150,-150)
    \Line(105,-45)(150,-135)
    \Line(195,-45)(150,-135)
    \Line(165,-45)(165,-105)
    \Line(135,-45)(165,-105)
  \end{picture}
}}}

\newcommand{\treec}{\scalebox{0.3}{{\parbox{0.5pc}{
  \begin{picture}(60,75) (135,-140)
    \SetWidth{2.0}
    \SetColor{Black}
    \Line(165,-75)(165,-120)
    \Line(135,-75)(165,-135)
    \Line(165,-105)(165,-150)
    \Line(195,-75)(165,-135)
  \end{picture}
}}}}

\newcommand{\treebcbis}{\scalebox{0.08}{{\parbox{0.5pc}{
  \begin{picture}(175,309) (78,-91)
    \SetWidth{5}
    \SetColor{Black}
    \Line(78,218)(120,115)
    \Line(120,116)(162,13)
    \Line(162,12)(162,-91)
    \Line(253,216)(162,13)
    \Line(162,218)(162,16)
    \Line(128,218)(120,116)
  \end{picture}
}}}}

\newcommand{\treecbsym}{\scalebox{0.08}{{\parbox{0.5pc}{
 \begin{picture}(175,309) (78,-91)
    \SetWidth{5}
    \SetColor{Black}
    \Line(78,218)(120,115)
    \Line(120,116)(162,13)
    \Line(162,12)(162,-91)
    \Line(253,216)(162,13)
    \Line(187,218)(163,116)
    \Line(140,218)(161,115)
    \Line(162,116)(162,16)
  \end{picture}
}}}}

\newcommand{\treecb}{\scalebox{0.08}{{\parbox{0.5pc}{
\begin{picture}(175,309) (78,-91)
    \SetWidth{5}
    \SetColor{Black}
    \Line(120,218)(120,115)
    \Line(78,218)(120,115)
    \Line(162,218)(120,115)
    \Line(120,116)(162,13)
    \Line(162,12)(162,-91)
    \Line(253,216)(162,13)
  \end{picture}
}}}}

\newcommand{\treecbbis}{\scalebox{0.08}{{\parbox{0.5pc}{
\begin{picture}(175,309) (78,-91)
    \SetWidth{5}
    \SetColor{Black}
    \Line(78,218)(120,115)
    \Line(120,116)(162,13)
    \Line(162,12)(162,-91)
    \Line(253,216)(162,13)
    \Line(162,218)(162,16)
    \Line(199,217)(207,114)
  \end{picture}
}}}}  

 \newcommand{\treed}{\scalebox{0.075}{{\parbox{0.5pc}{ 
\begin{picture}(272,343) (134,-119)
    \SetWidth{5}
    \SetColor{Black}
    \Line(406,222)(270,14)
    \Line(315,224)(270,15)
    \Line(225,222)(270,17)
    \Line(270,-119)(270,14)
    \Line(134,224)(271,15)
  \end{picture}
}}}}

  \def\surjection{\,{\scalebox{0.6}{
  \begin{picture}(418,184) (287,-236)
    \SetWidth{1.0}
    \SetColor{Black}
    \Line(376,-191)(376,-236)
    \Line(376,-191)(287,-71)
    \Line(354,-160)(328,-71)
    \Line(376,-191)(445,-71)
    \Line(393,-160)(358,-71)
    \Line(370,-101)(390,-71)
    \Line(411,-130)(413,-71)
    \Line(376,-191)(548,-71)
    \Line(420,-160)(481,-71)
    \Line(460,-102)(515,-71)
    \Line(376,-191)(586,-162)
    \Line(585,-71)(625,-102)
    \Line(586,-161)(705,-71)
    \Line(585,-162)(626,-102)
    \Line(616,-71)(625,-102)
    \Line(625,-102)(656,-71)
    \Text(314,-76)[lb]{\Large{\Black{$3$}}}
    \Text(528,-76)[lb]{\Large{\Black{$3$}}}
    \Text(565,-78)[lb]{\Large{\Black{$[4]$}}}
    \Text(597,-78)[lb]{\Large{\Black{$[1]$}}}
    \Text(630,-76)[lb]{\Large{\Black{$1$}}}
    \Text(671,-76)[lb]{\Large{\Black{$3$}}}
    \Text(337,-78)[lb]{\Large{\Black{$[4]$}}}
    \Text(369,-76)[lb]{\Large{\Black{$1$}}}
    \Text(394,-78)[lb]{\Large{\Black{$[3]$}}}
    \Text(425,-76)[lb]{\Large{\Black{$2$}}}
    \Text(459,-78)[lb]{\Large{\Black{$[4]$}}}
    \Text(490,-76)[lb]{\Large{\Black{$1$}}}
  \end{picture}
}}\,}

\nc{\ignore}[1]{{}}
\nc{\mrm}[1]{{\rm #1}}
\nc{\dirlim}{\displaystyle{\lim_{\longrightarrow}}\,}
\nc{\invlim}{\displaystyle{\lim_{\longleftarrow}}\,}
\nc{\vep}{\varepsilon} \nc{\ep}{\epsilon}
\nc{\sigmat}{\widetilde\sigma}
\nc{\ostar}{\overline{*}}

\nc{\mchar}{\mrm{Char}}
\nc{\Hom}{\mrm{Hom}}
\nc{\id}{\mrm{id}}

\nc{\remark}{\noindent{\bf{Remark:}}}
\nc{\remarks}{\noindent{\bf{Remarks:}}}

 \nc{\delete}[1]{}
 \nc{\grad}[1]{^{({#1})}}
 \nc{\fil}[1]{_{#1}}

\nc{\BA}{{\Bbb A}} \nc{\CC}{{\Bbb C}} \nc{\DD}{{\Bbb D}}
\nc{\EE}{{\Bbb E}} \nc{\FF}{{\Bbb F}} \nc{\GG}{{\Bbb G}}
\nc{\HH}{{\Bbb H}} \nc{\LL}{{\Bbb L}} \nc{\NN}{{\Bbb N}}
\nc{\PP}{{\Bbb P}} \nc{\QQ}{{\Bbb Q}} \nc{\RR}{{\Bbb R}}
\nc{\TT}{{\Bbb T}} \nc{\VV}{{\Bbb V}} \nc{\ZZ}{{\Bbb Z}}
\nc{\Cal}[1]{{\mathcal {#1}}}
\nc{\mop}[1]{\mathop{\hbox {\rm #1} }\nolimits}
\nc{\mopg}[1]{\mathop{\hbox {\bf #1} }\nolimits}
\nc{\smop}[1]{\mathop{\hbox {\eightrm #1} }\nolimits}
\nc{\mopl}[1]{\mathop{\hbox {\rm #1} }\limits}
\nc{\frakg}{{\frak g}}
\nc{\g}[1]{{\frak {#1}}}
\def \restr#1{\mathstrut_{\textstyle |}\raise-8pt\hbox{$\scriptstyle #1$}}
\def \srestr#1{\mathstrut_{\scriptstyle |}\hbox to
  -1.5pt{}\raise-4pt\hbox{$\scriptscriptstyle #1$}}
\nc{\wt}{\widetilde}
\nc{\wh}{\widehat}
\nc{\un}{\hbox{\bf 1}}
\nc{\redtext}[1]{\textcolor{red}{\tt #1}}
\nc{\bluetext}[1]{\textcolor{blue}{#1}}
\nc{\comment}[1]{[[{\tt {#1}}]] }
\nc{\R}{\mathbb R}
\nc\fleche[1]{\mathop{\hbox to #1 mm{\rightarrowfill}}\limits}
\def\semi{\mathrel{\times}\kern -.85pt\joinrel\mathrel{\raise
    1.4pt\hbox{${\scriptscriptstyle |}$}}}
\nc{\np}{/\hskip -2.3mm\pi}
\nc{\snp}{/\hskip -1.8mm\pi}

\begin{document}


\title[A closed formula for a discrete Magnus Expansion]
      {The tridendriform structure of a discrete Magnus expansion}

\author{Kurusch Ebrahimi-Fard}
\address{Instituto de Ciencias Matem\'aticas,
		C/ Nicol\'as Cabrera, no.~13-15, 28049 Madrid, Spain.
		On leave from Univ.~de Haute Alsace, Mulhouse, France.}
         \email{kurusch@icmat.es, kurusch.ebrahimi-fard@uha.fr}         
         \urladdr{www.icmat.es/kurusch}

\author{Dominique Manchon}
\address{Univ.~Blaise Pascal,
         	C.N.R.S.-UMR 6620,
         	63177 Aubi\`ere, France.}       
         \email{manchon@math.univ-bpclermont.fr}
         \urladdr{http://math.univ-bpclermont.fr/$\sim$manchon/}

\date{April 23, 2013}

\begin{abstract}
The notion of trees plays an important role in Butcher's B-series. More recently, a refined understanding of algebraic and combinatorial structures underlying the Magnus expansion has emerged thanks to the use of rooted trees. We follow these ideas by further developing the observation that the logarithm of the solution of a linear first-order finite-difference equation can be written in terms of the Magnus expansion taking place in a pre-Lie algebra. By using basic combinatorics on planar reduced trees we derive a closed formula for the Magnus expansion in the context of free tridendriform algebra. The tridendriform algebra structure on word quasi-symmetric functions permits us to derive a discrete analogue of the Mielnik--Pleba\'nski--Strichartz formula for this logarithm.
\end{abstract}

\maketitle

\begin{quote}
{\small{{\bf{key words}}: Magnus expansion; $B$-series; trees; pre-Lie algebra; tridendriform algebra, Rota--Baxter algebra; word quasi-symmetric functions; surjections.}}
\end{quote}


\tableofcontents


\section{Introduction}
\label{sect:intro}

In many areas of the mathematical sciences linear initial value problems (IVP) play an essential role. Recall that such a linear IVP basically consists of a first order linear differential equation
\begin{equation}
\label{eq:ivp}
	\dot{Y}(t)= A(t)Y(t),
\end{equation}
together with the initial value $Y(0) = Y_0$. The function $A(t)$ may be matrix or operator valued. It is common to write the solution of such an IVP in terms of the time-ordered exponential, $Y(t) = {\mathsf{T}}\!\exp\bigl(\int_0^t A(s)ds\bigr)Y_0$. Indeed, using the definition of the time-ordering operator ${\mathsf{T}}$ at distinct times $s_1,\ldots, s_n$
$$
	{\mathsf{T}}[U_1(s_1) \cdots U_n(s_n)]:=U_{\sigma(1)}(s_{\sigma(1)}) \cdots U_{\sigma(n)}(s_{\sigma(n)}),
$$
where $\sigma$ is the unique permutation such that $s_{\sigma(1)} < \cdots < s_{\sigma(n)}$, the function $Y(t)$ results as the formal solution
\begin{equation}
\label{eq:Torder} 
	{\mathsf{T}}\!\exp\Bigl(h\int_0^t A(s)ds\Bigr)Y_0 = Y_0\un + \sum_{n>0} \frac{h^n}{n!} \int_{[0,t]^n} {\mathsf{T}}[A(t_1) \cdots A(t_n)]dt_1 \cdots dt_n Y_0
\end{equation}
of the linear integral equation
\begin{equation}
\label{eq:fixpoint} 
	Y(t)=Y_0+ h\int_0^tA(s)Y(s)ds
\end{equation}
corresponding to (\ref{eq:ivp}). We have introduced the formal parameter $h$ for convenience. The first few terms are
\begin{equation}
\label{eq:TorderExplicit} 
	 Y(t)=  \Big(\un + h\int_0^t \!\!\!\!A(x_1) dx_1
	   + h^2 \int_0^t \!\!\!\!A(x_1)\!\!\int_0^{x_1} \!\!\!\! A(x_2) dx_2dx_1
	   + h^3\int_0^t \!\!\!\!A(x_1)\!\!\int_0^{x_1} \!\!\!\! A(x_2) \!\! \int_0^{x_2} \!\!\!\!A(x_3) dx_3dx_2dx_1 + \cdots\Big)Y_0 .
\end{equation}

The solution $Y(t)$ of (\ref{eq:ivp}) can also be written as a proper exponential. However, in general we can not expect that $Y(t) = \exp\bigl(h\int_0^t A(s)ds\bigr)Y_0$. Indeed, trying to re-arrange the coefficient of the second order term in $h$ yields
\begin{equation}
\label{eq:RB0}
	\frac{h^2}{2!}\int_0^t A(s)ds\int_0^t A(u)du 
	= \frac{h^2}{2!}\int_0^t \Big( \int_0^s A(u)du\Big)A(s)ds + \frac{h^2}{2!}\int_0^t A(s)\Big(\int_0^s A(u)du\Big)ds.
\end{equation}
Looking at the first term on the right-hand side, we see that the iterated integral is in ``bad" order, which means that the right-hand side does not add up to the $h^2$-order term in (\ref{eq:TorderExplicit}), namely $\big(\int_0^t A(s)ds\big)^2 \not\eq 2\int_0^t A(s)\int_0^s A(u)duds$. One may try to resolve this problem using the following simple ansatz. Introduce functions $\Omega_i(t)$, such that 
\begin{equation}
\label{eq:magnusCorr}
	Y(t) =  \exp\left(h\int_0^t A(s)ds + \sum_{i>1}h^i \Omega_i(t)\right)Y_0.
\end{equation}
Returning to (\ref{eq:RB0}), one verifies quickly that $\Omega_2(t) := -\frac{1}{2} \int_0^t \big[ \int_0^s A(u)du,A(s)\big]ds$ does the job -- up to order $h^2$. Indeed, observe that the unwanted term in (\ref{eq:RB0}) is canceled
\begin{eqnarray*}
	\frac{h^2}{2!}\Big(\int_0^t A(s)ds\Big)^2 -\frac{h^2}{2} \int_0^t \big[ \int_0^s A(u)du,A(s)\big]ds 
	&=& \frac{h^2}{2!}\int_0^t \int_0^s A(u)duA(s)ds + \frac{h^2}{2!}\int_0^t A(s)\int_0^s A(u)duds\\
	& &  -\frac{h^2}{2}\int_0^t \int_0^s A(u)duA(s)ds + \frac{h^2}{2}\int_0^t A(s)\int_0^s A(u)duds\\
	&=& h^2\int_0^t A(s)\int_0^s A(u)duds.
\end{eqnarray*}
It is clear that the introduction of this order $h^2$ correction term, $\Omega_2(t)$, will contribute at higher orders $h^n$, for $n>2$, which we have to take into account when calculating the function $\Omega_n(t)$. More generally, the function $\Omega_n(t)$ will depend on $\Omega_i(t)$, $0<i \leq n-1$.\\

The solution to this formidable rewriting problem was presented by Wilhelm Magnus in his seminal 1954 paper \cite{Magnus}, where he proposed for the logarithm of the time-ordered exponential, i.e., the logarithm of the formal series of iterated integrals (\ref{eq:TorderExplicit})
$$
	\Omega(hA)(t):=\log\Big({\mathsf{T}}\!\exp\bigl(h\int_0^t A(s)ds\bigr)\Big)
$$ 
(we assume $Y_0=\un$) a particular differential equation 
\begin{equation}
\label{eq:MagCl}
	\dot{\Omega}(hA)(t) 
	=hA(t) +  \sum_{n>0} \frac{B_n}{n!} ad_{\int_0^t \dot{\Omega}(hA)(s)ds}^{(n)}\big(hA(t)\big) 
	= \frac{ad_{\Omega(hA)}}{e^{ad_{\Omega(hA)}} - 1}\big(hA(t)\big),
\end{equation}
with $\Omega(hA)(0)=0$, and $\Omega(hA)(t)= \sum_{i>0}h^i \Omega_i(A;t)$,  $\Omega_1(A;t)=\int_0^tA(s)ds$. The $B_n$ are the Bernoulli numbers, and, as usual, the $n$-fold iterated Lie bracket is denoted by $ad^{(n)}_a(b):=[a,[a,\cdots [a,b]]\cdots]$. Note that in the literature one also finds the following notation $\dot{\Omega}(hA)(t)={\mathrm{d}}\!\exp^{-1}_{\Omega(hA)}(hA)$. See for instance Theorem 4 in \cite{BCOR}. The solution of the IVP (\ref{eq:ivp}) is then given by
\begin{equation}
\label{magclas}
	Y(t)=\exp\big(\Omega(hA)(t)\big)Y_0.
\end{equation}
Let us write down the first few terms of $\dot{\Omega}(hA)(s)$, following from Picard iteration
\begin{eqnarray*}
	h A(s) - h^2 \frac{1}{2}\left[\!\int_0^s\!\!\!\!A(x)dx ,A(s)\right]
		+h^3 \frac{1}{4}\left[\!\int_0^s\!\!\Big[\int_0^y\!\!\!\!A(x)dx ,A(y)\Big]dy,A(s)\right]
		 +h^3 \frac{1}{12}\left[\!\int_0^s\!\!\!\!A(x)dx,\Big[\!\int_0^s\!\!\!\!A(y)dy ,A(s)\Big]\right] + O(h^4).
\end{eqnarray*}

Magnus' seminal paper triggered much progress in both applied mathematics and physics. In the authoritative reference \cite{BCOR} the reader may find much more information. We should also remark that the presentation of Magnus' expansion given above is rather formal, since we have deliberately ignored aspects of convergence. The reason for this more algebraic approach to Magnus' expansion will become clear in the sequel. The principal purpose of it is to show that (\ref{eq:MagCl}) is just a particular case of a more general expansion that allows to solve fixed point equations like (\ref{eq:fixpoint}) in a far more general context than just the one given by IVPs.\\      

For the last 30 years or so, rooted trees play a central role in the theory of Butcher's B-series \cite{Butcher1, Butcher2,HW}. In the recent works \cite{Celledoni,ChaMu,Murua}, including the standard reference \cite{HLW}, the reader may find more details on the use of trees in numerical integration methods. Iserles and N{\o}rsett \cite{Iserles1} were the first to make extensive use of rooted trees to obtain a deeper understanding of the workings of Magnus' expansion. Included in the review article \cite{Iserles2} the reader can find a  comprehensive summary of the work of Iserles and N{\o}rsett. A very readable account on the use of rooted trees for Magnus' series can be found in \cite{Iserles3}.
 
\smallskip

In \cite{EM1,EM2} we started to explore the genuine pre-Lie algebra structure underlying Magnus' expansion. Two key observations form the basis for our approach. First, note that the basic building block in (\ref{eq:MagCl}), i.e., the Lie bracket with the integral operator on one side
$$
	\left[\int_0^s\!\!\!\!A(x)dx ,B(s)\right]  =: (A \rhd B)(s)
$$
defines a non-commutative binary product for, say, matrix valued functions $A,B$. It is easy to see that this product is non-associative. Indeed, it satisfies what is well-known as the {\textit{left pre-Lie identity}} \cite{Burde,ChaLiv,Manchon,Segal} 
$$
  (A \rhd B) \rhd C - A \rhd (B \rhd C)
    = (B \rhd A) \rhd C - B \rhd  (A \rhd C).
$$
This relation reflects the combination of integration by parts and the Jacobi identity. The second observation is based on expanding this Lie bracket, $\left[\int_0^s A(x)dx ,B(s)\right]=\int_0^s A(x)dx B(s) - B(s)\int_0^s A(x)dx$. One then shows by using integration by parts that the two binary non-associative products
\begin{equation}
\label{dend1}
	(A \succ B)(s) := \int_0^s A(x)dx B(s) \qquad\ (A \prec B )(s):= A(s) \int_0^s B(x)dx
\end{equation}
satisfy a non-commutative shuffle like structure \cite{AG2}, which is known as dendriform algebra \cite{Loday}. Going back to (\ref{eq:TorderExplicit}), we see that the iteration of the second product in (\ref{dend1}) yields the basic operation in the formal solution of (\ref{eq:ivp}). The iteration of the first operation analogously corresponds to the formal solution of $\dot{Y}(t)=Y(t)A(t)$. Hence, we see that these non-associative, non-commutative binary products reflect well the basic operations for solving linear IVPs.

With this in mind, let us return to (\ref{eq:MagCl}). In terms of the pre-Lie product $(\dot{\Omega}\rhd A)(s)=ad_{\int_0^s\! \dot{\Omega}(t)dt}(A(s))$, Magnus' series gains some transparency
\begin{equation}
\label{plm-intro}
	\dot{\Omega}(hA)=hA - h^2\frac{1}{2} A \rhd A 
		+ h^3\frac{1}{4}(A \rhd A) \rhd A
		+ h^3\frac{1}{12} A \rhd (A \rhd A) + \cdots
		=hA+\sum_{n>0} \frac{B_n}{n!} L_{\dot{\Omega}(A)\rhd}^{(n)}(hA).
\end{equation}
We denote by $L_{A\rhd}$ the left multiplication operator defined by $L_{A\rhd}(B):= A  \rhd B$. A similar approach applies to Fer's expansion \cite{EM1}. Note that the right-hand side of \eqref{plm-intro} already appeared in \cite{AG1} (where left pre-Lie algebras are called \textsl{chronological algebras}), but the dendriform structure is required to establish identity \eqref{plm-intro} itself \cite{EM1}. The pre-Lie picture is our starting point for the use of rooted trees in the exploration of Magnus' expansion. We would also like to mention the following references \cite{Chap1,Chap2,Chap3,ChaPat,CHNT}, which explore in depth pre-Lie aspects of Magnus' expansion. In \cite{gelfand} the Magnus expansion appears in the context of non-commutative symmetric functions.

One may wonder whether there is another expression for $\dot \Omega(A)$ in terms of the dendriform operations (\ref{dend1}) rather than using the pre-Lie product. In \cite{EM5} we gave a positive answer, which is based on a classical commutator free formula due to Mielnik--Pleba\'nski and Strichartz.

\begin{prop}[Mielnik--Pleba\'nski--Strichartz formula \cite{MielPleb,Strichartz}]\label{cor:mps2}
The function $\dot{\Omega}(A)(t)$ is given by the series of iterated integrals
\begin{equation}
\label{eq:mps2}
	\dot{\Omega}(A)(t)= \sum_{n > 0}  \sum_{\sigma\in S_n} \frac{(-1)^{d(\sigma)}}{\binom{n-1}{d(\sigma)}n} 
	\idotsint\limits_{0<u_n<\cdots <u_1<t}A(u_{\sigma_1})\cdots A(u_{\sigma_n})\,du_1\cdots du_n.
\end{equation}
\end{prop}
Here $S_n$ is the group of permutations of $n$ elements, and $d(\sigma)$ is the cardinality of the descent set $D(\sigma)\subset\{1,\ldots,n-1\}$ of the permutation $\sigma \in S_n$, i.e., the subset  of indices $i$ such that $\sigma(i)>\sigma(i+1)$. Unveiling the very dendriform nature of formula (\ref{eq:mps2}) requires the use of the free dendriform algebra with one generator (concretely described in terms of planar binary trees, or alternatively in terms of planar rooted trees via Knuth's rotation correspondence), as well as the Malvenuto--Reutenauer--Foissy bidendriform Hopf algebra $\hbox{\bf FQSym}$ of free quasi-symmetric functions.

\smallskip

As $\Omega(A)$ is a Lie element, we can use the Dynkin--Specht--Wever theorem (Theorem~8 on page 169 in \cite{jacobson}), so that we recover the formula in its original Lie algebraic setting
\begin{equation}
\label{eq:mpslie}
	\dot{\Omega}(A)(t)= \sum_{n > 0}  \sum_{\sigma\in S_n} \frac{(-1)^{d(\sigma)}}{\binom{n-1}{d(\sigma)}n^2 } 
		\idotsint\limits_{0<u_n<\cdots <u_1<t}
	[A(u_{\sigma_1}),[A(u_{\sigma_2}),\ldots [A(u_{\sigma_{n-1}},\,A(u_{\sigma_n})]\cdots]]\,du_1\cdots du_n.
\end{equation}

\smallskip

In the case of discrete analogues of differentiation and integration further refinement is needed. Recall that corresponding to the modified Leibniz rule for finite-difference operators, summation operators satisfy a modified integration by parts identity. The latter accounts for non-trivial diagonal terms. Therefore, replacing the Riemann integral by a Riemann sum operator, which we denote by $\Sigma$, yields three binary products
$$
	A \succ B := \Sigma(A) B \qquad\ A \prec B := A\Sigma(B)  \qquad\ A\cdot B:= AB.
$$
They are known to form a tridendriform algebra \cite{LR3}, which can be interpreted as a non-commutative quasi-shuffle like structure \cite{PaNoTh}. 

\medskip 

This paper is a continuation of our work \cite{EM5}. First we explore the Magnus expansion from the tridendriform algebra point of view, using planar reduced rooted trees. Then we aim at  a ``discrete analogue" of the Mielnik--Pleba\'nski--Strichartz formula (\ref{eq:mps2}), i.e., iterated integrals will be replaced by iterated sums. Contrarily to the continuous case, partial diagonals have to be taken into account. The relevant algebraic structure will be given by the one of tridendriform algebra, which is a natural refinement of the notion of dendriform algebra \cite{LR} proposed by J-L.~Loday and M.~Ronco in \cite{LR3}. The Malvenuto--Reutenauer--Foissy bidendriform Hopf algebra $\hbox{\bf FQSym}$ must then be replaced by the more refined \textsl{tridendriform Hopf algebra} $\hbox{\bf WQSym}$ of \textsl{word quasi-symmetric functions}, where the groups $S_n$ of permutations of $\{1,\ldots ,n\}$ are replaced by the sets $ST_n^r$ of surjective maps from $\{1,\ldots ,n\}$ onto $\{1,\ldots ,r\}$.

\smallskip

The free tridendriform algebra with one generator is concretely described in terms of planar reduced trees \cite{LR}, or alternatively in terms of planar rooted hypertrees via a suitable extension of Knuth's rotation correspondence \cite{EM4}. The tridendriform Hopf algebra $\hbox{\bf WQSym}$ can be traced back to F.~Hivert's PhD thesis \cite{Hivert}, in which he constructs the even larger Hopf algebra $\hbox{\bf MQSym}$ of \textsl{matrix quasi-symmetric functions}, which naturally contains $\hbox{\bf WQSym}$. A clear account of the associated tridendriform structure can be found in \cite{NT}. This object has also been thoroughly studied under the notation $\hbox{\bf ST}$ by E.~Burgunder and M.~Ronco in \cite{BR}.

\smallskip

The discrete Mielnik--Pleba\'nski--Strichartz formula splits into two versions according to whether one excludes the upper bound from the summation operator or not, see equations \eqref{dmps1} and \eqref{dmps2} respectively. Both look similar to \eqref{eq:mps2} once iterated integrals have been replaced with iterated sums, except that the notion of descent, extended from permutations to surjections, splits into a strict and a weak version, each of them giving its variant of the formula. The strict (resp.~weak) descent set of a surjection $\sigma:\{1,\ldots ,n\}\surj\{1,\ldots ,r\}$ is the set of indices $i\in\{1,\ldots ,n-1\}$ such that $\sigma(i)>\sigma(i+1)$ (resp. $\sigma(i)\ge \sigma(i+1)$). 

\smallskip

We note that as well as any dendriform algebra is naturally pre-Lie, any tridendriform algebra is naturally endowed with a structure of \textsl{post-Lie algebra} \cite{BGN}. The latter is a vector space together with a binary product $\rhd$ and a Lie bracket $[-,-]$ subject to compatibility axioms \cite{LunMun1, LunMun2,Vallette}. Recently, due to the work of Munthe-Kaas et al.~it became clear that post-Lie algebras play a central role in the theory of Lie group integrators on manifolds. It would be interesting to understand the post-Lie algebra structure underlying  the Magnus expansion by refining \eqref{plm-intro} for logarithms of solutions of discrete IVPs. We plan to address this problem in a forthcoming paper.\\

The paper is organized as follows: in Section \ref{sect:AlgPri} we review the notion of trees, and introduce the essential algebraic structures. In Section \ref{sect:MagExp} we give a detailed description of two ``Magnus elements", namely the logarithms of the solutions of two first-order linear tridendriform equations, corresponding to the two dendriform structures one can associate to a tridendriform algebra. After a reminder of the pre-Lie Magnus expansion, we give a tridendriform Magnus expansion of the two Magnus elements above in terms of planar reduced trees, when the tridendriform algebra is free. Finally in Section \ref{sect:PleMieSti}, relating the tridendriform algebra of sequences with $\hbox{\bf WQSym}$ and with the free tridenriform algebra, we give the discrete analogue of the Mielnik--Pleba\'nski--Strichartz formula.

\medskip

\noindent
{\bf{Acknowledgements:}} We would like to thank  A.~Lundervold, H.~Munthe-Kaas, E.~Burgunder, M.~Livernet, F.~Patras, M.~Ronco and J.-Y.~Thibon for discussions and remarks. We are thankful to the referees for their comments and suggestions. The first author is supported by a Ram\'on y Cajal research grant from the Spanish government, as well as the project MTM2011-23050 of the Ministerio de Econom\'{i}a y Competitividad. Both authors were supported by the CNRS (GDR Renormalisation), and by Agence Nationale de la Recherche, projet CARMA ANR-12-BS01-0017.\\


\section{Algebraic and combinatorial preliminaries}
\label{sect:AlgPri}

Throughout the paper, $k$ will stand for a field of characteristic zero. In this section we recall the notion of trees, as well as the relevant algebraic structures.


\subsection{Planar reduced trees}
\label{sect:trees}

Recall that a tree $t$ is a connected and simply connected graph made out of vertices and edges, the sets of which we denote by $V(t)$ and $E(t)$, respectively. A \textsl{planar reduced tree} is a finite oriented tree given an embedding in the plane, such that all vertices have two or more incoming edges, and exactly one outgoing edge. An edge can be internal (connecting two vertices) or external (with one loose end). The external incoming edges are the leaves. The root edge is the unique edge not ending in a vertex. For any planar reduced tree $t$, a partial order on the set of its vertices $V(t)$ is defined as follows: $u,v \in V(t)$, $u<v$ if and only if there is a path from the root of $t$ through $u$ up to $v$. A planar reduced tree is \textsl{binary} if all vertices have exactly two incoming edges.
$$
	\raise -4pt\hbox{$\Big\vert$}
	\qquad\
	\scalebox{1.5}{\tree}
	\qquad\
	\treeA
	\quad
	\treec
	\qquad
	\treeB
	\qquad\
	\treeE
	\quad
	\treebcbis
	\qquad
	\treeD
	\quad
	\treecbbis
	\qquad
	\treeC
	\quad
	\treebc
	\quad
	\treeG	
	\quad
	\treecbsym
	\qquad
	\treeF 
	\quad
	\treecb
	\qquad
	\treed\qquad\ldots
$$
We include the unique planar reduced tree without internal vertices, i.e., the single edge $\vert$, despite the fact that it is not binary in the strict sense. We denote by $T^{red}_{pl}$ (resp.~$\Cal T^{red}_{pl}$) the set (resp.~the linear span) of planar reduced trees. A simple grading for such trees is given in terms of the number of internal vertices. Alternatively, one can use the number of leaves. Above we listed all planar reduced trees up to four leaves. Observe that for any collection $(t_1,\ldots ,t_n)$ of planar reduced trees we can build up a new planar reduced tree via the grafting operation, $t:=\bigvee(t_1,\ldots,t_n)$, by considering the unique planar binary tree with one single vertex and $n$ leaves, and plugging $t_k$ on leaf number $k$,\, $k\in\{1,\ldots ,n\}$, from left to right.\\

Any planar reduced tree $t\not =\vert$ obviously can be expressed as $\bigvee(t_1,\ldots ,t_n)$ in a unique way. The grafting operation $\bigvee$ makes $T_{pl}^{red}$ the free generic magma with one generator and one operation in any arity $a\ge 2$. Notice that this product is of degree zero with respect to the leaf number grading. However we will adopt the grading $t\to|t|$ given by the number of leaves of $t$ minus one. We call the binary trees $\tau^{(n)}_r$, $\tau^{(n)}_l$ recursively defined by $\tau^{(0)}_{r}:= \vert=:\tau^{(0)}_{l}$ and $\tau^{(n+1)}_r:= \bigvee( |,\tau^{(n)}_r)$, $\tau^{(n+1)}_l:= \bigvee(\tau^{(n)}_l, |)$ right and left combs, respectively.  The following list includes the right and left combs up to order three
$$
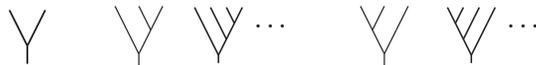
\scalebox{1.5}{\tree}
	\qquad\
	\treeA
	\quad
	\treeC\
	\cdots
	\qquad\
	\treeB
	\quad\
	\treeE\
	\cdots
$$ 
There is a partial order on $T^{red}_{pl}$ defined as follows: $t_1 \le t_2$ if $t_1$ can be obtained from $t_2$ by glueing some inner vertices together. Here glueing refers to shrinking an edge between two adjacent inner vertices until it becomes a new inner vertex. In particular, two comparable trees must have the same degree. The minimal elements are the trees with only one inner vertex, and the maximal elements are the planar binary trees. 
\begin{rmk}
Following a suggestion by Jean-Louis Loday\footnote{Private communication}, we have proved that a natural extension of D. Knuth's rotation correspondence settles a natural bijection from planar reduced trees onto {\rm planar rooted hypertrees\/}. As we won't use this fact here, we refer the reader to \cite{EM4} for details.
\end{rmk}


\subsection{Pre- and Post-Lie algebras}
\label{ssect:pL}

Recall that a left {\sl{pre-Lie algebra}} $(\Cal P, \rhd)$ is a $k$-vector space $\Cal P$ equipped with an operation $\rhd: \Cal P \otimes \Cal P \rightarrow \Cal P$ subject to the following relation
$$
	(a \rhd b) \rhd c - a \rhd ( b \rhd c) = (b \rhd a) \rhd c - b \rhd (a \rhd c).
$$
The Lie bracket following from antisymmetrization in $\Cal P$, $[a,b]:=a \rhd b - b \rhd a$, satisfies the Jacobi identity. See e.g.~\cite{Manchon} for a survey on pre-Lie algebras. Recall from Chapoton and Livernet \cite{ChaLiv} that the basis of the free pre-Lie algebra in one generator can be expressed in terms of undecorated, non-planar rooted trees. See also \cite{AG1, Segal} for other descriptions of free pre-Lie algebra. 

A natural example of pre-Lie algebra is given in terms of a differentiable manifold $M$ endowed with a flat torsion-free connection. The corresponding covariant derivation operator $\nabla$ on the space $\big(\chi(M),\rhd\big)$ of vector fields on $M$ gives it a left pre-lie algebra structure defined by $a\rhd b=\nabla_ab$, by virtue of the two equalities
\begin{equation*}
	\nabla_ab-\nabla_ba=[a,b], \hskip 15mm \nabla_{[a,b]}=[\nabla_a,\nabla_b]
\end{equation*}
which express the vanishing of torsion and curvature respectively. For $M=\R^n$ with the standard flat connection we have for
$a=\sum_{i=1}^na_i\partial_i$ and $b=\sum_{i=1}^nb_i\partial_i$
\begin{equation*}
	a\rhd b=\sum_{i=1}^n\left(\sum_{j=1}^na_j(\partial_jb_i)\right)\partial_i.
\end{equation*}

A left {\sl{post-Lie algebra}} $(\Cal Q, \rhd,[-,-]))$ is a Lie algebra $\Cal Q$ with Lie bracket $[-,-]$, together with another operation $\diamond : \Cal Q \otimes \Cal Q \rightarrow \Cal Q$ subject to the following two compatibility relations
\begin{eqnarray}
	&&a\diamond [b,c]=[a\diamond b,\,c] + [a,\,b\diamond c], \label{postLie}\\
	&&[a,b] \diamond c=a\diamond (b\diamond c)-(a\diamond b)\diamond c
							-b\diamond (a\diamond c)+(b\diamond a)\diamond c.
\end{eqnarray}
Note that a pre-Lie algebra is a post-Lie algebra with vanishing Lie bracket. The natural geometric example of a post-Lie algebra is given in terms of a connection which is flat and has constant torsion. See \cite{LunMun2,LunMun3} for details.


\subsection{Rota--Baxter algebras}
\label{sssect:dpse}
 
Recall that a {\sl{Rota--Baxter algebra}} is a $k$-algebra $\Cal A$ endowed with a $k$-linear map $T: \Cal A \to \Cal A$ that satisfies the relation
\begin{equation}
\label{RB}
    T(a)T(b) = T\big(T(a)b + aT(b) + \theta ab\big),
\end{equation}
where $\theta \in k$ is a fixed parameter \cite{Baxter}. The map $T$ is called a {\sl Rota--Baxter operator of weight $\theta$\/}. The map $\widetilde{T}:=-\theta \textrm{id}-T$ also is a weight $\theta$ Rota--Baxter map. Both images $T(\Cal A)$ and $\wt{T}(\Cal A)$ are subalgebras in $\Cal A$. One may think of (\ref{RB}) as a generalized integration by parts identity. Indeed, a simple example of Rota--Baxter algebra is given by the classical integration by parts rule showing that the ordinary Riemann integral is a weight zero Rota--Baxter map. Other examples can be found for instance in \cite{EM1,EM2,EM3}.


\subsection{Tridendriform algebras}
\label{ssect:DendAlg}

We introduce the notion of {\sl{tridendriform algebra}}~\cite{LR3} over $k$, which is a $k$-vector space $\Cal D$ endowed with three bilinear operations $\prec$, $\succ$ and $\cdot$, subject to the seven following axioms
\begin{eqnarray}
	(a\prec b)\prec c  &=& a\prec(b \prec c + b \succ c+b\cdot c),       	\label{A1}\\
  	(a\succ b)\prec c  &=& a\succ(b\prec c),   				\label{A2}\\
   	a\succ(b\succ c)  &=& (a \prec b + a \succ b+b\cdot c)\succ c,    	\label{A3}\\
	(a\cdot b)\cdot c&=&a\cdot(b\cdot c),	\label{A4}\\
	(a\succ b)\cdot c&=&a\succ(b\cdot c),	\label{A5}\\
	(a\prec b)\cdot c&=&a\cdot(b\succ c),	\label{A6}\\
	(a\cdot b)\prec c&=&a\cdot(b\prec c).	\label{A7}
\end{eqnarray}
Axioms (\ref{A1})-(\ref{A7}) imply that for $a,b \in \Cal D$ the composition
\begin{equation}
\label{dendassoc}
	a * b := a \prec b + a \succ b+a\cdot b
\end{equation}	
defines an associative product. At first this may look puzzling, but further below we will see that finite differences provide a natural and elucidating example, showing that these axioms encode the modified integration by parts formula for summation operators. 

A {\sl{dendriform algebra}} is defined by setting the product $\cdot$ to zero in the above axioms. Hence, the rules of a dendriform algebra  are given in terms of axioms (\ref{A1})-(\ref{A3}) alone, without the $\cdot$ term. Note, for example, that this reduced set of rules encodes integration by parts for the Riemann integral. 

However, any tridendriform algebra $(\Cal D,\prec,\succ,\cdot)$ gives rise to two ordinary dendriform algebras $\Cal{D}_L:=(\Cal D,\preceq,\succ)$ and $\Cal{D}_R:=(\Cal D,\prec,\succeq)$ with $\preceq:=\prec+\,\cdot$ and $\succeq:=\succ+\,\cdot$. Recall that dendriform algebras, hence tridendriform algebras as well, are at the same time pre-Lie algebras. Indeed, the two following products inherited from the dendriform structure $\Cal{D}_R$
\begin{equation}
\label{def:prelie}
    a \rhd b:= a\succeq b-b\prec a,
    \hskip 12mm
    a \lhd b:= a\prec b-b\succeq a
\end{equation}
are left pre-Lie and right pre-Lie, respectively. That is, we have
\begin{eqnarray*}
    (a\rhd b)\rhd c-a\rhd(b\rhd c)&=& (b\rhd a)\rhd c-b\rhd(a\rhd c),\\
    (a\lhd b)\lhd c-a\lhd(b\lhd c)  &=& (a\lhd c)\lhd b-a\lhd(c\lhd b).
\end{eqnarray*}
The Lie brackets following from the associative operation (\ref{dendassoc}) and the pre-Lie operations (\ref{def:prelie}) all define the same Lie algebra. The same holds of course \textsl{mutatis mutandis} for the other dendriform algebra $\Cal{D}_L$, giving rise to two other pre-Lie products which will be denoted by $\urhd$ and $\ulhd$. Note that the associative product (\ref{dendassoc}) is the same for both dendriform structures.

\smallskip 

For any tridendriform algebra $\Cal D$ we denote by $\overline{\Cal D} = \Cal{D} \oplus k.\un$ the corresponding dendriform algebra augmented by a unit $\un$, with the following rules
\begin{equation*}
    a \prec \un := a =: \un \succ a
    \hskip 12mm
    \un \prec a = a \succ \un=\un\cdot a=a\cdot \un:=0,
\end{equation*}
implying $a*\un=\un*a=a$. Note that the equality $\un*\un=\un$ makes sense, but that $\un \prec \un$, $\un \cdot \un$ and $\un \succ \un$ are not defined. \\

Now suppose that the tridendriform algebra $\Cal D$ is complete with respect to the topology given by a decreasing filtration $\Cal D=\Cal{D}^1\supset \Cal{D}^2\supset \Cal{D}^3\supset\cdots$, compatible with the dendriform structure, in the sense that $\Cal{D}^p\prec \Cal{D}^q\subset \Cal{D}^{p+q}$, $\Cal{D}^p\succ \Cal{D}^q\subset \Cal{D}^{p+q}$ and $\Cal{D}^p\cdot \Cal{D}^q\subset \Cal{D}^{p+q}$ for any $p,q\ge 1$. In the unital algebra we can then define the exponential and logarithm map in terms of the associative product~(\ref{dendassoc})
$$
	\exp^*(x):=\sum_{n \geq 0} x^{*n}/n!  
	\quad\ {\rm{resp.}} \quad\ 
	\log^*(\un+x):=-\sum_{n>0}(-1)^nx^{*n}/n. 
$$
Let $L_{a \succ} \left( b \right) := a \succ b$ and $R_{\succ b} \left( a \right) := a \succ b$. Note that $L_{a \succ} L_{b \succ} = L_{(a \ast b) \succ}$. We recursively define the set of tridendriform words in $\overline{\Cal{D}}$ for fixed elements $x_1,\ldots, x_n \in \Cal{D}$, $n \in \mathbb{N}$ by
 \allowdisplaybreaks{
\begin{eqnarray*}
    w^{(0)}_{\prec}(x_1,\ldots, x_n) &:=& \un =:w^{(0)}_{\succ}(x_1,\ldots, x_n) \\
    w^{(n)}_{\prec}(x_1,\ldots, x_n) &:=& x_1 \prec \bigl(w^{(n-1)}_\prec(x_2,\ldots, x_n)\bigr)\\
    w^{(n)}_{\succ}(x_1,\ldots, x_n) &:=& \bigl(w^{(n-1)}_\succ(x_1,\ldots, x_{n-1})\bigr)\succ x_n.
\end{eqnarray*}}
In case that $x_1=\cdots = x_n=x$ we simply write $w^{(n)}_{\prec}(x,\ldots, x)=x^{(n)}_{\prec}$ and 
 $w^{(n)}_{\succ}(x,\ldots, x)= x^{(n)}_{\succ}$.\\

\noindent Our main example of a tridendriform algebra comes from the following simple observation. One verifies easily that any associative Rota--Baxter algebra $\Cal A$ of weight $\theta$ gives rise to a tridendriform algebra as follows
\allowdisplaybreaks{
\begin{eqnarray}
\label{RBtridend}
    a \prec b := aT(b),\quad\ a\succ b := T(a)b,\quad\ a\cdot b:=\theta ab.
\end{eqnarray}}
The corresponding associative and left pre-Lie products are explicitly given for $a,b \in \Cal A$ by
\begin{eqnarray}
\label{RBdoublePreLie}
	a*b&=&T(a)b+aT(b)+\theta ab,\\ 
	 a \rhd b& =& [T(a),b] + \theta ab,\\
	 a\urhd b&=& [T(a),b]-\theta ba.
\end{eqnarray}
Note that in a commutative Rota--Baxter algebra with weight $\theta\not =0$, the pre-Lie products are still nontrivial although the Lie brackets vanish. This leads to the classical Spitzer identity \cite{EM1,EM2}. By omitting the $\theta$-terms the product
\begin{equation}
\label{RBpostLie} 
	a \diamond b :=   a\succ b - b \prec a = [T(a),b] 
\end{equation} 
yields a post-Lie algebra structure (\ref{postLie}) on $\Cal A$ with respect to the Lie bracket defined in terms of the third tridendriform product \cite{BGN}. 

\smallskip 

The $\theta$-term on the right hand side of (\ref{RB}), respectively the product $\cdot$ in the definition of the tridendriform algebra, is necessary, for instance, when we replace the Riemann integral by a Riemann-type summation operator. This becomes evident once we recall the modified Leibniz rule for the finite difference operator $\delta(f)(x):=f(x+1)-f(x)$ on a suitable class of functions 
$$
	\delta(fg) = \delta (f) g + f\delta(g) + \delta(f)\delta(g). 
$$
The corresponding summation operator
\begin{equation}
\label{sum}
    \Sigma(f)(x):= \sum_{n=0}^{[x]-1} f(n)
\end{equation}
verifies the weight $\theta=1$ Rota--Baxter relation
\begin{equation*}
     \Sigma(f) \Sigma(g) = \Sigma\big(\Sigma(f)  g +f \Sigma(g)+fg\big).
\end{equation*}
See further below in subsection \ref{ssect:tridendSeq} for more details. More generally for finite Riemann sums
\begin{eqnarray}
\label{Riemsum}
    T_\theta(f)(x) := \sum_{n = 0}^{[x/\theta]-1} \theta f(n\theta)
\end{eqnarray}
where $\theta$ is a positive real number and $[-Ð]$ is the floor function, we find that $T_\theta$ satisfies the weight $\theta$ Rota--Baxter relation
$$ 
	T_\theta(f)T_\theta(g)=
	T_\theta\big(T_\theta(f)g + fT_\theta(g)+\theta fg\big).
$$


\subsubsection{Tridendriform algebra structure on planar reduced trees}
\label{sssect:DendAlgTree}

In \cite{LR3} it was shown that the linear span ${\Cal T'}^{red}_{pl}$ of planar reduced trees different from $\vert$ generates the free tridendriform algebra in one generator. Starting from taking $\vert$ as a unit for the associative product $*$, the three products for two trees $s=\bigvee(s_1,\ldots,s_n)$ and  $t=\bigvee(t_1,\ldots, t_p)$ are given recursively by
\begin{eqnarray*}
	s\prec t&=&\bigvee(s_1,\ldots,s_{n-1},s_n*t),\\
	s\succ t&=&\bigvee(s*t_1,t_2,\ldots,t_p),\\
	s\cdot t&=&\bigvee(s_1,\ldots,s_{n-1},s_n*t_1,t_2,\ldots,t_p).
\end{eqnarray*}
The tree $\vert$ can be taken as the unit for the corresponding augmented dendriform algebra. For any  collection of trees $(t_1,\ldots,t_n)$ we easily derive
\begin{equation}
\label{eq:tridend1}
	\bigvee(t_1,t_2)=t_1\succ\raise 3pt\hbox{$\scalebox{0.8}\tree$}\prec t_2,
\end{equation}
as well as
\begin{eqnarray}
\label{eq:tridend}
	\bigvee(t_1,\ldots ,t_n)&=&t_1\succ\bigvee(\vert,t_2,\ldots,t_n)\nonumber\\
					  &=&t_1\succ\raise 3pt\hbox{$\scalebox{0.8}\tree$}\cdot\bigvee(t_2,\ldots,t_n)\nonumber\\
					  &=&(t_1\succ\raise 3pt\hbox{$\scalebox{0.8}\tree$})\cdot\ \cdots\ \cdot(t_{n-2}\succ
					  \raise 3pt\hbox{$\scalebox{0.8}\tree$}) \cdot (t_{n-1}\succ\raise 3pt\hbox{$\scalebox{0.8}\tree$}\prec t_n)
\end{eqnarray}
for $n\ge 2$. We have omitted parentheses in the second line in the computation above, by virtue of Axiom \eqref{A5}. The freeness property of $({\Cal T'}^{red}_{pl},\succ,\prec,\cdot)$ implies that for any tridendriform algebra $\Cal D$ and any $a\in \Cal D$ there is a unique morphism $F_a : { \Cal T'}^{red}_{pl} \to \Cal D$. Using (\ref{eq:tridend1}) and (\ref{eq:tridend}), this morphism can be described recursively. Indeed, starting from $F_a(\raise 3pt\hbox{$\scalebox{0.8}\tree$}):=a$ we have 
\begin{eqnarray*}
	F_a(t)	&=&	F_a\left(\bigvee(t_1,\ldots,t_n)\right) \\
			&=&\big(F_a(t_1)\succ a\big)\cdot\ \cdots\ \cdot \big(F_a(t_{n-2})\succ a\big)\cdot\big(F_a(t_{n-1})\succ a\prec F_a(t_n)\big),
\end{eqnarray*}
for any $n\ge 2$, as is easily seen from \eqref{eq:tridend}.


\section{Linear tridendriform equations and the pre-Lie Magnus expansion}
\label{sect:MagExp}

In this section we abstract (\ref{eq:fixpoint}) into linear fixed point equations in the complete filtered tridendriform algebra ${h\Cal A[[h]]}$, augmented by a unit $\un$, where $(\Cal A,\prec,\succ,\cdot)$ is any tridendriform algebra. For $a \in \Cal A$, let $X=X(ha)$ and $\overline X=\overline X(ha)$ be solutions of
\begin{eqnarray}
			X &=& \un + h a \prec X			\label{eq:dendri-linear1} \\
	  \overline X &=& \un + h a \preceq \overline X,	\label{eq:dendri-linear2}
\end{eqnarray}
respectively. Equation \eqref{eq:dendri-linear1}, resp.  \eqref{eq:dendri-linear2}, is understood to take place in the unital dendriform algebra $\overline{h\Cal A_R[[h]]}$, resp. $\overline{h\Cal A_L[[h]]}$. Their formal solutions are
\begin{eqnarray}
	X&=& \sum_{n\ge 0} (h a)^{(n)}_{\prec}= \un +  h a 
		+ h^2 a \prec a 
		+ h^3 a \prec \left( a \prec a \right) 
		+ h^4 a \prec \left( a \prec \left( a \prec a \right)\right) +\cdots,\\ 
	\overline X&=& \sum_{n\ge 0} (h a)^{(n)}_{\preceq}= \un +  h a 
		+ h^2 a \preceq a 
		+ h^3 a \preceq \left( a \preceq a \right) 
		+ h^4 a \preceq \left( a \preceq \left( a \preceq a \right)\right) +\cdots .
\end{eqnarray}


\subsection{The pre-Lie Magnus expansion}
\label{ssect:pLMag}

In \cite{EM1,EM2} we have given a general formula, the \textsl{pre-Lie Magnus expansion}, for the logarithm of such linear dendriform equations in terms of the left pre-Lie product. Applying this to the two dendriform structures above, we obtain, with the notations of Paragraph \ref{ssect:DendAlg}

\begin{thm} \label{thm:pLMagnus}
{\rm{(\cite{EM1,EM2})}}
The elements $\Omega' = \log^*(X(h a))$ and $\overline\Omega' = \log^*(\overline X(h a))$ in $\Cal A[[h]]$ satisfy respectively the two recursive formulas
\begin{eqnarray}
	\Omega' &=& \frac{L_{\Omega' \rhd}}{e^{L_{\Omega' \rhd} } - 1} (h a) 
	   = \sum_{m \geqslant 0} \frac{B_m}{m!}  L_{\Omega' \rhd}^{(m)} (h a), \label{pLMag1}\\
	   \overline\Omega' &=& \frac{L_{\overline\Omega' \underline\rhd}}{e^{L_{\overline\Omega' \underline\rhd} } - 1} (h a) 
	   = \sum_{m \geqslant 0} \frac{B_m}{m!}  L_{\overline\Omega' \underline\rhd}^{(m)} (h a)\label{pLMag2}
\end{eqnarray}
where $B_m$ is the $m$-th Bernoulli number. The first few terms are
\begin{eqnarray*}
			\Omega'(h a)&=&h  a - h^2 \frac 12 a\rhd a 
					+ h^3 \left(\frac 14 (a\rhd a)\rhd a + \frac 1{12} a\rhd(a\rhd a)\right)+\cdots,\\
	   \overline\Omega'(h a)&=&h  a - h^2 \frac 12 a\urhd a 
					+ h^3 \left(\frac 14 (a\urhd a)\urhd a + \frac 1{12} a\urhd(a\urhd a)\right)+\cdots				
\end{eqnarray*}
\end{thm}
Recall (\ref{eq:magnusCorr}), the terms beyond order $h$ in $\Omega'(h a)$ are needed to eliminate the unwanted terms when calculating the $X(h a)=\exp^*(h a + \sum_{m \geqslant 1} \frac{B_m}{m!}  L_{\Omega' \rhd}^{(m)} (h a))$.


\subsection{The post-Lie Magnus expansion}
\label{ssect:postLieMag}

Splitting in strands the pre-Lie product, we see from a tridendriform point of view, that the pre-Lie multiplication operator in (\ref{pLMag1}) decomposes into
$$
	L_{a \rhd} = L_{a \diamond} + L_{a \cdot },
$$ 
where
$$
	L_{a \diamond}(b) := a \diamond b = a \succ b - b \prec a
$$ 
and the left multiplication $L_{a \cdot }(b):=a \cdot b$. Similarly the other pre-Lie multiplication operator in (\ref{pLMag2}) decomposes into
$$
	L_{a \underline\rhd} = L_{a \diamond} - R_{\cdot a},
$$
with the right multiplication $R_{\cdot a}(b):=b \cdot a$. Recall that the vector space $\Cal A$ together with two bilinear binary operations $\diamond$ and the Lie bracket following from the $\cdot$ product, $[a,b]_{\cdot}:=a \cdot b - b \cdot a$, defines a post-Lie algebra (\ref{postLie}). This post-Lie algebra is very particular, as it comes with two compatible pre-Lie structures, namely $a\rhd b=a\diamond b +a\cdot b$ and $a\underline\rhd b=a\diamond b-b\cdot a$. Splitting the pre-Lie product $\rhd=\diamond +\cdot$ in \eqref{pLMag1}, or analogously $\underline\rhd=\diamond -\cdot^{\smop{op}}$ in \eqref{pLMag2}, yields  refinements of the pre-Lie Magnus expansions of the elements $\Omega'$ and $\overline{\Omega'}$ described in the previous paragraph, which would be very interesting to understand in greater detail. A forthcoming paper will be devoted to exploring the post-Lie structure of the Magnus expansion.


\subsection{A closed formula for the logarithm}
\label{ssect:closedLog}

In this paragraph, expanding the pre-Lie product, we give an explicit expression of $\log^*(X)$ and $\log^*(\overline X)$ in the completed free tridendriform algebra in one generator. We can safely set the dummy parameter $h$ to $1$ here, thanks to completeness.\\

Recall that a leaf of a planar binary tree is a \textsl{descent} if it is not the leftmost one and if it is pointing up to the left \cite{Chap2}. For example, take the following right and left combs
$$
	\treeA
	\quad
	\treeC\
	\qquad\
	\treeB
	\quad\
	\treeE\
	\cdots
$$ 
The first tree has one descent and the second has two descents. The last two trees have no descents. We extend this notion to planar reduced trees in two different ways as follows: a leaf is a \textsl{descent} if it is not the leftmost one, and if it is not the rightmost edge above a vertex. A \textsl{strict descent} is a descent which is moreover the leftmost edge above some vertex.

\begin{thm}\label{thm:closed}
The elements $\Omega' = \log^*(X)$ and $\overline\Omega' = \log^*(\overline X)$ in the completion of the free unital tridendriform algebra $\Cal A=\Cal T_{pl}^{red}$ are given by the formulas
\begin{eqnarray}
	\Omega' &=& \sum_{n > 0}\sum_{t \in \Cal{T}_{pl}^{red} \atop |t|=n} \frac{(-1)^{d(t)}}{n\binom{n-1}{d(t)}}\ t,\label{eq:dendMagnus1}\\
	\overline\Omega' &=& \sum_{n > 0} \sum_{t \in \Cal{T}_{pl}^{red} \atop |t|=n} \frac{(-1)^{\overline d(t)}}{n\binom{n-1}{\overline d(t)}} \ t.\label{eq:dendMagnus2}
\end{eqnarray}
where $\overline d(t)$ (resp. $d(t)$) denotes the number of descents (resp. strict descents) of $t$, and $|t|$ its number of leaves minus one.
\end{thm}

\begin{proof}
Both statements will be derived from \cite[Corollary 6]{EM5}. There is a unique dendriform algebra morphism $F_L$ (resp. $F_R$) from the free dendriform algebra ${\Cal T_{pl}'}^{bin}$ to $\Cal A_L$ (resp. $\Cal A_R$) such that $F_L(\raise 3pt\hbox{$\scalebox{0.8}\tree$})=F_R(\raise 3pt\hbox{$\scalebox{0.8}\tree$})=\raise 3pt\hbox{$\scalebox{0.8}\tree$}$.

\begin{lem}\label{lem:flfr}
For any $t\in{{\Cal T}_{pl}'}^{bin}$ we have
\begin{eqnarray}
	F_L(t)  &=&\sum_{t'\le t \atop \overline d(t')=d(t)}t',	\label{fl}\\
	F_R(t) &=&\sum_{t'\le t \atop d(t')=d(t)}t'.			\label{fr}
\end{eqnarray}
\end{lem}

\begin{proof}
Recall that the notions of descent and strict descent coincide for planar binary trees. Lemma \ref{lem:flfr} is obviously true for $t=\raise 3pt\hbox{$\scalebox{0.8}\tree$}$. Let us prove it by induction on the degree $|t|$. Remark first that, in any planar reduced tree, shrinking an inner edge does not change the number of descents if and only if this edge points up to the right. Similarly, shrinking an inner edge does not change the number of strict descents if and only if this edge points up to the left. Shrinking any other inner edge ``in between" will simultaneously increase the number of descents and decrease the number of strict descents by one. Hence the right-hand side of \eqref{fl}, resp. \eqref{fr}, is the sum of all planar reduced trees which can be obtained from $t$ by repeatedly glueing two vertices together, provided they are linked by an edge pointing up to the right, resp. up to the left. Recall that any planar binary tree writes $t=t_1\vee t_2=t_1\succ\raise 3pt\hbox{$\scalebox{0.8}\tree$}\prec t_2$ in a unique way. We can compute, using the induction hypothesis
\allowdisplaybreaks{
\begin{eqnarray*}
	F_L(t)	&=& F_L(t_1\vee t_2)=F_L(t_1\succ\raise 3pt\hbox{$\scalebox{0.8}\tree$}\prec t_2)\\
			&=& F_L(t_1)\succ\raise 3pt\hbox{$\scalebox{0.8}\tree$}\preceq F_L(t_2)\\
			&=& \sum_{t'_1\le t_1 \atop \overline d(t'_1)=d(t_1)}\ \sum_{t'_2\le t_2 \atop \overline d(t'_2)= d(t_2)}
				t'_1\succ\raise 3pt\hbox{$\scalebox{0.8}\tree$}\preceq t'_2\\
			&=& \sum_{t'_1\le t_1 \atop \overline d(t'_1)= d(t_1)}\ \sum_{t'_2\le t_2 \atop \overline d(t'_2)= d(t_2)}
				t'_1\succ\raise 3pt\hbox{$\scalebox{0.8}\tree$}\prec t'_2
				+\sum_{t'_1\le t_1 \atop \overline d(t'_1)= d(t_1)}\ \sum_{t'_2\le t_2 \atop \overline d(t'_2)= d(t_2)}
				t'_1\succ\raise 3pt\hbox{$\scalebox{0.8}\tree$}\cdot t'_2\\
			&=&\sum_{t'\le t \atop \overline d(t')=d(t)}t'.
\end{eqnarray*}}

The computation of $F_R(t)$ is done similarly using strict descents. As an example, we have
\allowdisplaybreaks{
\begin{eqnarray*}
	F_L(\treeB)&=&\treeB ,\\
	F_L(\treeA)&=&\treeA+\treec\hskip 6mm,\\
	F_R(\treeB)&=&\treeB+\treec\hskip 6mm,\\
	F_R(\treeA)&=&\treeA.
\end{eqnarray*}}
\end{proof}

\textit{End of proof of Theorem \ref{thm:closed}}: Corollary 6 of \cite{EM5} applied to $\Cal A_L$ and $\Cal A_R$ reads:
\begin{eqnarray*}
	\overline\Omega' &=& \sum_{n > 0} \sum_{t \in \Cal{T}_{pl}^{red} \atop |t|=n} \frac{(-1)^{d(t)}}{n\binom{n-1}{d(t)}} F_L(t),\\
	\Omega' &=& \sum_{n > 0}  \sum_{t \in \Cal{T}_{pl}^{red} \atop |t|=n} \frac{(-1)^{ d(t)}}{n\binom{n-1}{ d(t)}} F_R(t).
\end{eqnarray*}
Applying Lemma \ref{lem:flfr} then immediately yields Theorem \ref{thm:closed}.
\end{proof}


\section{A discrete analogue of the Mielnik--Pleba\'nski--Strichartz formula}
\label{sect:PleMieSti}


\subsection{A tridendriform structure on surjections \cite{NT, BR}}
\label{ssect:surjdend}

Let $E$ be some finite set. For any $f:E\to\{1,2,\ldots\}$ there exists a unique positive integer $r$ and a unique surjective map $\mop{std}(f):E\to\hskip -3mm\to\{1,\ldots ,r\}$ such that $f(i)<f(j)$ if and only if $\mop{std}(f)(i)<\mop{std}(f)(j)$. This surjective map is the \textsl{standardization} of $f$. For example the standardization of the map $(2,7,4,1,4):\{1,2,3,4,5\}\to\{1,2,3,4,5,6,7\}$ is $(2,4,3,1,3):\{1,2,3,4,5\}\surj\{1,2,3,4\}$. A map $f:E\to\{1,2,\ldots\}$ is \textsl{standard} if its image is equal to some initial interval $\{1,\ldots,r\}$. For any $f:\{1,\ldots,n\}\to\{1,2,\ldots\}$ and $g:\{1,\ldots,p\}\to\{1,2,\ldots\}$, their juxtaposition $fg:\{1,\ldots,n+p\}\to\{1,2,\ldots\}$ is defined by $fg(r)=f(r)$ for $r=1,\ldots,n$ and $fg(n+r)=g(r)$ for $r=1,\ldots,p$. Juxtaposition is obviously associative. \\

For any positive integers $1\le r\le n$, let $\mop{ST}_n^r$ be the set of surjective maps from $\{1,\ldots ,n\}$ onto $\{1,\ldots ,r\}$, and set $\mop{ST}_n:=\bigcup_{r=1}^n \mop{ST}_n^r$. Let $\mopg{WQSym}=\bigoplus_{n\ge 1}\mopg{ST}_n$ be the graded vector space such that $\mop{ST}_n$ freely generates the homogeneous component $\mopg{ST}_n$ for any $n\ge 1$. Three bilinear products on $\mopg{WQSym}$ are defined as follows
\begin{eqnarray*}
	f\succ g &=&\sum_{{\scriptstyle\smop{std}(F)=f,\,\smop{std}(G)=g,\atop {\scriptstyle 
				FG\smop{ standard} \atop \smop{ max}(F)<\smop{max}(G)}}} FG,\\
	f\prec g &=&\sum_{{\scriptstyle\smop{std}(F)=f,\,\smop{std}(G)=g,\atop {\scriptstyle 
				FG\smop{ standard} \atop \smop{ max}(F)>\smop{max}(G)}}} FG,\\
	f\cdot g &=&\sum_{{\scriptstyle\smop{std}(F)=f,\,\smop{std}(G)=g,\atop {\scriptstyle 
				FG\smop{ standard} \atop \smop{ max}(F)=\smop{max}(G)}}} FG.
\end{eqnarray*}
\begin{prop}[\cite{NT, PR, BR}]
$(\mopg{WQSym},\,\prec,\succ,\cdot)$ is a graded tridendriform algebra.
\end{prop}

\begin{proof}
The reader is invited to check the seven tridendriform axioms. For example
\begin{eqnarray*}
	(f*g)\succ h=f\succ(g\succ h)	&=&\sum_{{\scriptstyle\smop{std}(F)=f,\,\smop{std}(G)=g,\,\smop{std}(H)=h,\atop {\scriptstyle 
	FGH~\smop{ standard} \atop \smop{ max}(F),\,\smop{max}(G)<\smop{max}(H)}}}FGH,\\
	(f\cdot g)\cdot h=f\cdot(g\cdot h)&=&\sum_{{\scriptstyle\smop{std}(F)=f,\,\smop{std}(G)=g,\,\smop{std}(H)=h,\atop {\scriptstyle 
	FGH~\smop{ standard} \atop \smop{ max}(F)=\smop{max}(G)=\smop{max}(H)}}}FGH,\\
	(f\succ g)\prec h=f\succ(g\prec h)&=&\sum_{{\scriptstyle\smop{std}(F)=f,\,\smop{std}(G)=g,\,\smop{std}(H)=h,\atop {\scriptstyle 
	FGH~\smop{ standard} \atop\smop{ max}(F)<\smop{max}(G)>\smop{max}(H)}}}FGH,
\end{eqnarray*}
and similarly for the four remaining ones. Compatibility with the grading is obvious. A complete proof can be found, e.g., in \cite[Chap. 2]{PR}.
\end{proof}


\subsection{Planar reduced trees and surjections}
\label{ssect:pbtPerm}

The material presented in this paragraph is mostly borrowed from \cite{LR3} and \cite{BR}. A bijective correspondence between surjections and \textsl{planar reduced trees with levels} is described as follows:  a planar reduced tree with $r$ levels is a planar reduced tree $t$ with, say, $m$ internal vertices and $n+1$ leaves together with a surjective nonincreasing map $\varphi$ from the poset of its internal vertices onto $\{1,\ldots,r\}$. Such a tree admits a graphical realization by drawing the internal vertices at the prescribed levels, with level $1$ being the top one and level $r$ being the deepest one. Any planar reduced tree with levels $(t,\varphi)$ gives rise to several such trees $(t_1,\varphi_1),\ldots ,(t_k,\varphi_k)$, where $t=\bigvee(t_1,\ldots ,t_k)$ and $\varphi_i$ is the standardized restriction of the map $\varphi$ to the internal vertices of $t_i$.\\

To any such tree $(t,\varphi)$ we can associate a surjection $\sigma_{t,\varphi}:\{1,\ldots,n\}$ onto $\{1,\ldots, r\}$ as follows: $\sigma_{t,\varphi}(i)$ is the level of the internal vertex $u_i$ situated between leaves $l_i$ and $l_{i+1}$ (the leftmost being the first and the rightmost being number $n+1$).  This correspondence $P$ is a bijection, the inverse of which is recursively given as follows: the surjection $\sigma:\{1,\ldots n\}\to\hskip -3mm\to\{1,\ldots, r\}$ reaches its maximal value $r$ a certain number of times, say $k-1$. It gives then rise to $k$ sequences of integers, possibly with repetitions, in $\{1,\ldots, r-1\}$. Some of them can of course be empty. By ``standardizing" the integers in each sequence, they form a surjection. For instance $(341324134113)$ gives the four sequences $(3)$, $(132)$, $(13)$ and $(113)$ which, after standardizing, give the four surjections $(1)$, $(132)$, $(12)$ and $(112)$. The grafting $\bigvee$ of the $k$ trees (in the order given above) gives the underlying tree of $P^{-1}(\sigma)$, and the original surjection is used to determine the levels of each vertex, namely $\varphi(u_j)=\sigma(j)$.\\
\begin{align*}
\surjection &\\
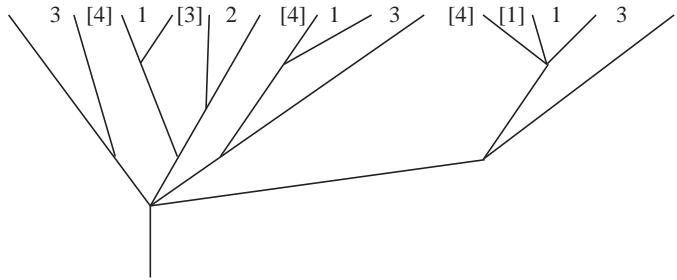
	{\smop{Planar reduced tree with levels associated to the surjection } \scriptstyle (341324134113).
\atop
	\smop{All descents, indicated with bracketed numbers, are strict except the [1] on the right}.}&
\end{align*}
Forgetting the levels, we thus obtain for any positive integer $n$ a surjective map $\Psi:\mop{ST}_n\surj (T_{pl}^{red})_n$. A \textsl{descent}, resp. a \textsl{strict descent} of a surjection $f$ in $\mop{ST}_n$ is an integer $j\in\{1,\ldots, n-1\}$ such that $f(j)\ge f(j+1)$, resp.~$f(j)> f(j+1)$. These notions match with the corresponding notion for planar reduced trees. In fact, given a planar reduced tree with levels, any descent (resp.~any strict descent) gives rise to a corresponding descent (resp. strict descent) of the associated surjection, and vice-versa. As an obvious corollary we have for any $f\in\mop{ST}$:
\begin{equation}
\label{descentes}
	d(f)=d\big(\Psi(f)\big),\hskip 8mm \overline d(f)=\overline d\big(\Psi(f)\big).
\end{equation}
Dualizing, we can consider the injective linear map
\begin{eqnarray*}
	\Psi^*:\Cal T_{pl}^{red}	  &\longrightarrow\ \mopg{WQSym}\\
						t &\longmapsto\ \sum_{\Psi(f)=t} f.
\end{eqnarray*}
\begin{prop}
	The map $\Psi^*$ is a tridendriform algebra morphism.
\end{prop}
\begin{proof}
See e.g., P.~Palacios and M.~Ronco (\cite{PR}, Theorem 48 and Corollary 49).
\end{proof}


\subsection{The tridendriform algebra of sequences}
\label{ssect:tridendSeq}

Let $\Cal A$ be the vector space of sequences \noindent $\NN=\{0,1,2,3,\ldots\}\to \Cal B$, where $\Cal B$ is some unital associative algebra (not necessarily commutative). Let $D:\Cal A\to \Cal A$ be the difference operator, defined by
\begin{equation}
	Df(N):=f(N+1)-f(N).
\end{equation}
A right inverse for $D$ is the summation operator
\begin{equation}
	Sf(N):=\sum_{r=0}^{N-1}f(r).
\end{equation}
We have indeed $DSf(N)=f(N)$ and $SDf(N)=f(N)-f(0)$. It is well-known that $S$ verifies the weight one Rota--Baxter property
\begin{equation}
	SfSg=S\big((Sf)g+f(Sg)+ fg\big).
\end{equation}
Thus $(\Cal A,\,\prec,\succ,\cdot)$ is a tridendriform algebra, with
\begin{equation*}
	f\prec g:=fS(g),\hskip 8mm f\succ g:=S(f)g,\hskip 8mm f\cdot g:=fg.
\end{equation*}
Any surjective map $\sigma:\{1,\ldots ,n\}\surj\{1,\ldots ,r\}$ defines a \textsl{partial diagonal} $T_\sigma\subset \NN^n$ as follows
\begin{equation}
	T_\sigma:=\{(s_1,\ldots,s_n)\in\NN^n,\, s_i >s_j\Leftrightarrow \sigma(i)<\sigma(j) 
	\hbox{ and }s_i=s_j\Leftrightarrow \sigma(i)=\sigma(j),\,j=1,\ldots ,n\}.
\end{equation}
The inversion of order is purely conventional. This yields a partition of $\NN^n$, namely $\NN^n=\coprod_{\sigma\in\smop{ST}_n}T_\sigma$. The same holds if $\NN^n$ is replaced with the hypercube $\{0,\ldots ,N-1\}^n$, provided $N\ge n$, yielding to a partition of $\{0,\ldots ,N-1\}^n$ with the partial diagonals $T_\sigma(N)$. If this last condition is not verified, some $T_\sigma(N)$'s can be empty. By convention all $T_\sigma(N)$'s are empty for $N=0$.

The following lemma is the discrete analogue of splitting the shuffle relations for iterated integrals into two ``half-shuffle" parts

\begin{lem}\label{decoupage}
For any $\sigma\in\mop{ST}_n$ and $\tau\in\mop{ST}_m$ and for any $N\ge n+m$ we have
\begin{equation}
	T_\sigma(N)\times T_\tau(N)=\coprod_{{\scriptstyle\smop{std}(F)
	=f,\,\smop{std}(G)=g,\atop \scriptstyle FG\smop{ standard}}}T_{FG}(N).
\end{equation}
Moreover this product of two partial diagonals splits into three parts
\allowdisplaybreaks{
\begin{eqnarray}
	T_\sigma(N)\ltimes T_\tau(N)&=&\coprod_{{\scriptstyle\smop{std}(F)
		=f,\,\smop{std}(G)=g,\atop \scriptstyle {FG\smop{ standard} \atop \smop{ max}(F)>\smop{max}(G)}}}T_{FG}(N),\\
	T_\sigma(N)\rtimes T_\tau(N)&=&\coprod_{{\scriptstyle\smop{std}(F)
		=f,\,\smop{std}(G)=g,\atop \scriptstyle {FG\smop{ standard} \atop \smop{ max}(F)<\smop{max}(G)}}}T_{FG}(N),\\
	T_\sigma(N)\bowtie T_\tau(N)&=&\coprod_{{\scriptstyle\smop{std}(F)
		=f,\,\smop{std}(G)=g,\atop \scriptstyle {FG\smop{ standard} \atop \smop{ max}(F)=\smop{max}(G)}}}T_{FG}(N),\\
\end{eqnarray}}
where:
\allowdisplaybreaks{
\begin{eqnarray*} 
	T_\sigma(N)\ltimes T_\tau(N)&:=&\{(\underline s,\underline t)\in T_\sigma(N)\times T_\tau(N),\, \mop{min}\underline 
	s<\mop{min} \underline t\},\\
	T_\sigma(N)\rtimes T_\tau(N)&:=&\{(\underline s,\underline t)\in T_\sigma(N)\times T_\tau(N),\, \mop{min}\underline 
	s>\mop{min} \underline t\},\\
	T_\sigma(N)\bowtie T_\tau(N)&:=&\{(\underline s,\underline t)\in T_\sigma(N)\times T_\tau(N),\, \mop{min}\underline 
	s=\mop{min} \underline t\}.
\end{eqnarray*}}
Lemma extends to any $N\in\{0,\ldots,n+m-1\}$ provided one accepts empty components in the right-hand sides of the equalities.
\end{lem}

\begin{proof}
Take any $\underline u=(\underline s,\underline t)=(u_1,\ldots,u_{n+m})$ in $T_\sigma(N)\times T_\tau(N)$, and order the $u_j$'s from largest to smallest. This uniquely defines a surjection $\gamma\in\mop{ST}_{n+m}$, by sending the largest $u_j$'s on $1$, the second largest ones on $2$ and so on. The standardization of $F=\gamma\restr{\{1,\ldots, n\}}$ (resp. $G=\gamma\restr{\{n+1,\ldots,n+m\}}$) is equal to $\sigma$ (resp. $\tau$). Then obviously $\underline u\in T_\gamma(N)$, which proves the first assertion. If moreover $\mop{min}\underline s<\mop{min}\underline t$, then $\mop{max}F>\mop{max}G$, and similarly if $\mop{min}\underline s>\mop{min}\underline t$ or $\mop{min}\underline s=\mop{min}\underline t$, which proves the lemma.
\end{proof}

\noindent
Now let $a\in\Cal A$, and let $\wt F_a:\mopg{WQSym}\to \Cal A$ the linear map defined for any $\sigma\in\mop{ST}_n$ by
\begin{equation}
	\wt F_a(\sigma)(N):=D\left(N\mapsto \sum_{T_\sigma(N)}a(s_1)\cdots a(s_n)\right).
\end{equation}
\begin{thm}\label{thm:tridend}
The map $\wt F_a:\mopg{WQSym}\to \Cal A$ defined above is a tridendriform algebra morphism, and we have
\begin{equation}
	F_a=\wt F_a\circ\Psi^*,
\end{equation}
where $F_a:\Cal T_{pl}^{red}\to\Cal A$ is the unique tridendriform algebra morphism such that $F_a(\raise 3pt\hbox{$\scalebox{0.8}\tree$})=a$. In other words, we have the following commutative diagram of tridendriform algebra morphisms:
\diagramme{
\xymatrix{
\Cal T_{pl}^{red} \ar[r]^{F_a}\ar[d]^{\Psi^*} & \Cal A	\\
\mopg{WQSym} \ar[ur]_{\wt F_a}}
}

\end{thm}

\begin{proof}
By direct computation: take $\sigma\in\mop{ST}_n$ and $\tau\in\mop{ST}_m$. Then, using Lemma \ref{decoupage},
\allowdisplaybreaks{
\begin{eqnarray*}
	S\big(\wt F_a(\sigma\prec\tau)\big)(N) &=& \sum_{{\scriptstyle\smop{std}(F)=\sigma,\,\smop{std}(G)=\tau,\atop 
	{\scriptstyle FG\smop{ standard} \atop \smop{ max}(F)>\smop{max}(G)}}} S\big(\wt F_a(FG)\big)(N)\\
								  &=& \sum_{{\scriptstyle\smop{std}(F)=\sigma,\,\smop{std}(G)=\tau,\atop 
								  {\scriptstyle FG\smop{ standard} \atop \smop{ max}(F)>\smop{max}(G)}}}
								  \sum_{T_{FG}(N)}a(s_1)\cdots a(s_{n+m})\\
								  &=& \sum_{T_\sigma(N)\ltimes T_\tau(N)}a(s_1)\cdots a(s_{n+m})\\
								  &=& S\Big(\wt F_a(\sigma)S\big(\wt F_a(\tau)\big)\Big)(N)\\
								  &=&S\big(\wt F_a(\sigma)\prec\wt F_a(\tau)\big)(N).
\end{eqnarray*}}
The conclusion follows by applying $D$ to both sides. The computation for $\succ$ and $\cdot$ is completely similar.
\end{proof}

\begin{cor}
The element $\Omega'(a)=\log^*\big(X(a)\big)$ in $\Cal A$, where $X(a)$ is the solution of the linear dendriform equation $X(a)=\un+a\prec X(a)$, is formally given by the series
\begin{equation*}
	\Omega'(a)=\sum_{n>0}\sum_{\sigma\in\smop{ST}_n}
	\frac{(-1)^{d(\sigma)}}{n{n-1\choose d(\sigma)}}\wt F_a(\sigma).
\end{equation*}
Similarly, the element $\overline\Omega'(a)=\log^*\big(\overline X(a)\big)$ in $\Cal A$, where $\overline X(a)$ is the solution of the linear dendriform equation $\overline X(a)=\un+a\preceq \overline X(a)$, is formally given by the series

\begin{equation*}
	\overline\Omega'(a)=\sum_{n>0}\sum_{\sigma\in\smop{ST}_n}
		\frac{(-1)^{\overline d(\sigma)}}{n{n-1\choose \overline d(\sigma)}}\wt F_a(\sigma).
\end{equation*}
\end{cor}

\begin{proof}
This is an immediate consequence of \cite[Corollary 6]{EM5}, Theorem \ref{thm:tridend} and \eqref{descentes}.
\end{proof}


\subsection{The discrete Mielnik--Pleba\'nski--Strichartz formula}
\label{ssect:MPSformula}

\begin{cor}[Discrete Mielnik--Pleba\'nski--Strichartz formula]
The ``discrete Magnus elements" $\Omega(a)=S\big(\Omega'(a)\big)$ and $\overline \Omega(a)=S\big(\overline \Omega'(a)\big)$ are given by the series
\begin{eqnarray}
	\Omega(a)(N)&=&\sum_{n>0}\sum_{\sigma\in\smop{ST}_n}\frac{(-1)^{d(\sigma)}}{n{n-1\choose d(\sigma)}}
					\sum_{T_	\sigma(N)}a(s_1)\cdots a(s_n),		\label{dmps1}\\
	\overline \Omega(a)(N)&=&\sum_{n>0}\sum_{\sigma\in\smop{ST}_n}\frac{(-1)^{\overline d(\sigma)}}{n{n-1\choose
	 \overline d(\sigma)}}\sum_{T_\sigma(N)}a(s_1)\cdots a(s_n).		\label{dmps2}
\end{eqnarray}
\end{cor}

\begin{rmk}
{\rm
Contrarily to what happens in the continuous case \cite{EM5}, the rewriting of $\Omega(a)$ and $\overline\Omega(a)$ as Lie elements is not obvious. Hence, the representation of \eqref{dmps1} and \eqref{dmps2} in terms of Lie brackets is rather involved. This is related to the fact that the pre-Lie products $\rhd$ and $\underline\rhd$ in a Rota--Baxter algebra cannot be expressed in terms of the Lie bracket and the Rota--Baxter operator alone, unless the weight $\theta$ is equal to zero. It is at this point where the post-Lie structure enters the picture. We plan to address this in detail in a forthcoming paper.}
\end{rmk}

\bigskip


\end{document}